\documentclass[a4paper,12pt]{amsart}
\usepackage{amssymb}
\usepackage{bbm}
\usepackage[all]{xy}

\normalsize

\newlength{\aufzleft}
\newenvironment{aufz}{\begin{list}{}{\setlength{\listparindent}{0pt}\setlength{\itemsep}{\topsep}\setlength{\labelwidth}{3.2ex}\setlength{\aufzleft}{\labelsep}\addtolength{\aufzleft}{\labelwidth}\setlength{\leftmargin}{\aufzleft}}}{\end{list}}
\newenvironment{equi}{\begin{list}{}{\setlength{\listparindent}{0pt}\setlength{\itemsep}{\topsep}\setlength{\labelwidth}{4.1ex}\setlength{\aufzleft}{\labelsep}\addtolength{\aufzleft}{\labelwidth}\setlength{\leftmargin}{\aufzleft}}}{\end{list}}

\newtheoremstyle{bracket}{1ex}{2ex}{\rm}{}{\bfseries}{}{0.8em}{\thmnumber{(#2)}}
\newtheoremstyle{thm}{1ex}{2ex}{\itshape}{}{\bfseries}{}{0.9em}{\thmnumber{(#2) }\thmname{#1}\thmnote{ #3}}
\newtheoremstyle{example}{1ex}{2ex}{\rm}{}{\bfseries}{}{0.8em}{\thmnumber{(#2)}\thmname{ #1}}

\theoremstyle{bracket}
\newtheorem{no}{}[section]

\theoremstyle{thm}
\newtheorem{prop}[no]{Proposition}
\newtheorem{lemma}[no]{Lemma}
\newtheorem{cor}[no]{Corollary}
\newtheorem{thm}[no]{Theorem}
\newtheorem*{thmno}{Theorem}

\theoremstyle{example}
\newtheorem{qu}[no]{Questions}

\newcommand{\dfgl}{\mathrel{\mathop:}=}
\newcommand{\N}{\mathbbm{N}}
\newcommand{\Z}{\mathbbm{Z}}
\newcommand{\Q}{\mathbbm{Q}}
\newcommand{\F}{\mathbbm{F}}

\DeclareMathOperator{\ke}{Ker}
\DeclareMathOperator{\inte}{Int}
\DeclareMathOperator{\cinte}{CInt}
\DeclareMathOperator{\degsupp}{degsupp}
\DeclareMathOperator{\nzd}{Nzd}

\newcommand{\snf}{\renewcommand{\thefootnote}{*}\footnotetext{The author was supported by the Swiss National Science Foundation.}}

\begin{document}

\title{Graded integral closures\snf}
\author{Fred Rohrer}
\address{Universit\"at T\"ubingen, Fachbereich Mathematik, Auf der Morgenstelle 10, 72076 T\"u\-bingen, Germany}
\email{rohrer@mail.mathematik.uni-tuebingen.de}

\begin{abstract}
It is investigated how graded variants of integral and complete integral closures behave under coarsening functors and under formation of group algebras.
\end{abstract}

\maketitle


\section*{Introduction}

Let $G$ be a group, let $R$ be a $G$-graded ring, and let $S$ be a $G$-graded $R$-algebra. (Throughout, monoids, groups and rings are understood to be commutative, and algebras are understood to be commutative, unital and associative.) We study a graded variant of (complete) integral closure, defined as follows: We denote by $\inte(R,S)$ (or $\cinte(R,S)$, resp.) the $G$-graded sub-$R$-algebra of $S$ generated by the {\em homogeneous} elements of $S$ that are (almost) integral over $R$ and call this the (complete) integral closure of $R$ in $S$. If $R$ is entire (as a $G$-graded ring, i.e., it has no {\em homogeneous} zero-divisors) we consider its graded field of fractions $Q(R)$, i.e., the $G$-graded $R$-algebra obtained by inverting all nonzero {\em homogeneous} elements, and then $\inte(R)=\inte(R,Q(R))$ (or $\cinte(R)=\cinte(R,Q(R))$, resp.) is called the (complete) integral closure of $R$. These constructions behave similar to their ungraded relatives, as long as we stay in the category of $G$-graded rings. But the relation between these constructions and their ungraded relatives, and more generally their behaviour under coarsening functors, is less clear; it is the main object of study in the following.

For an epimorphism of groups $\psi\colon G\twoheadrightarrow H$ we denote by $\bullet_{[\psi]}$ the $\psi$-coarsening functor from the category of $G$-graded rings to the category of $H$-graded rings. We ask for conditions ensuring that $\psi$-coarsening commutes with relative (complete) integral closure, i.e., $\inte(R,S)_{[\psi]}=\inte(R_{[\psi]},S_{[\psi]})$ or $\cinte(R,S)_{[\psi]}=\cinte(R_{[\psi]},S_{[\psi]})$, or -- if $R$ and $R_{[\psi]}$ are entire -- that $\psi$-coarsening commutes with (complete) integral closure, i.e., $\inte(R)_{[\psi]}=\inte(R_{[\psi]})$ or $\cinte(R)_{[\psi]}=\cinte(R_{[\psi]})$. Complete integral closure being a more delicate notion than integral closure, it is not astonishing that the questions concerning the former are harder than the other ones. Furthermore, the case of integral closures of entire graded rings in their graded fields of fractions turns out to be more complicated than the relative case, because $Q(R)_{[\psi]}$ almost never equals $Q(R_{[\psi]})$, hence in addition to the coarsening we also change the overring in which we form the closure.

The special case $H=0$ of parts of these questions was already studied by several authors. Bourbaki (\cite[V.1.8]{ac}) treats torsionfree groups $G$, Van Geel and Van Oystaeyen (\cite{vgvo}) consider $G=\Z$, and Swanson and Huneke (\cite[2.3]{sh}) discuss the case that $G$ is of finite type. Our main results, generalising these partial results, are as follows.

\begin{thmno}[1]
Let $\psi\colon G\twoheadrightarrow H$ be an epimorphism of groups and let $R$ be a $G$-graded ring.

a) If $\ke(\psi)$ is contained in a torsionfree direct summand of $G$ then $\psi$-coarsening commutes with relative (complete) integral closure.

b) Suppose that $R$ is entire. If $G$ is torsionfree, or if $\ke(\psi)$ is contained in a torsionfree direct summand of $G$ and the degree support of $R$ generates $G$, then $\psi$-coarsening commutes with integral closure.
\end{thmno}

The questions above are closely related to the question of how (complete) integral closure behaves under formation of group algebras. If $F$ is a group, there is a canonical $G\oplus F$-graduation on the algebra of $F$ over $R$; we denote the resulting $G\oplus F$-graded ring by $R[F]$. We ask for conditions ensuring that formation of graded group algebras commutes with relative (complete) integral closure, i.e., $\inte(R,S)[F]=\inte(R[F],S[F])$ or $\cinte(R,S)[F]=\cinte(R[F],S[F])$, or -- if $R$ is entire -- that formation of graded group algebras commutes with (complete) integral closure, i.e., $\inte(R)[F]=\inte(R[F])$ or $\cinte(R)[F]=\cinte(R[F])$. Our main results are the following.

\begin{thmno}[2]
Let $G$ be a group and let $R$ be a $G$-graded ring. Formation of graded group algebras over $R$ commutes with relative (complete) integral closure. If $R$ is entire then formation of graded group algebras over $R$ commutes with (complete) integral closure.
\end{thmno}

It is maybe more interesting to consider a coarser graduation on group algebras, namely the $G$-graduation obtained from $R[F]$ by coarsening with respect to the canonical projection $G\oplus F\twoheadrightarrow G$; we denote the resulting $G$-graded $R$-algebra by $R[F]_{[G]}$ and call it the coarsely graded algebra of $F$ over $R$. We ask for conditions ensuring that formation of coarsely graded group algebras commutes with relative (complete) integral closure, i.e., $\inte(R,S)[F]_{[G]}=\inte(R[F]_{[G]},S[F]_{[G]})$ or $\cinte(R,S)[F]_{[G]}=\cinte(R[F]_{[G]},S[F]_{[G]})$, or -- if $R$ and $R[F]_{[G]}$ are entire -- that formation of coarsely graded group algebras commutes with (complete) integral closure, i.e., $\inte(R)[F]_{[G]}=\inte(R[F]_{[G]})$ or $\cinte(R)[F]_{[G]}=\cinte(R[F]_{[G]})$. Ungraded variants of these questions (i.e., for $G=0$) for a torsionfree group $F$ were studied extensively by Gilmer (\cite[\S 12]{gil}). On use of Theorems 1 and 2 we will get the following results.

\begin{thmno}[3]
Let $G$ and $F$ be groups and let $R$ be a $G$-graded ring. Formation of the coarsely graded group algebra of $F$ over $R$ commutes with relative (complete) integral closure if and only if $F$ is torsionfree. If $R$ is entire and $F$ is torsionfree then formation of the coarsely graded group algebra of $F$ over $R$ commutes with integral closure.
\end{thmno}

Some preliminaries on graded rings, coarsening functors, and algebras of groups are collected in Section 1. Relative (complete) integral closures are treated in Section 2, and (complete) integral closures of entire graded rings in their graded fields of fractions are treated in Section 3. Our notation and terminology follows Bourbaki's \textit{\'El\'ements de math\'ematique.}

\smallskip

Before we start, a remark on notation and terminology may be appropriate. Since we try to never omit coarsening functors (and in particular forgetful functors) from our notations it seems conceptually better and in accordance with the general yoga of coarsening to not furnish names of properties of $G$-graded rings or symbols denoting objects constructed from $G$-graded rings with additional symbols that highlight the dependence on $G$ or on the graded structure. For example, if $R$ is a $G$-graded ring then we will denote by $\nzd(R)$ (instead of, e.g., $\nzd_G(R)$) the monoid of its \textit{homogeneous} non-zerodivisors, and we call $R$ entire (instead of, e.g., $G$-entire) if $\nzd(R)$ consists of all homogeneous elements of $R$ different from $0$. Keeping in mind that in this setting the symbol ``$R$'' always denotes a $G$-graded ring (and never, e.g., its underlying ungraded ring), this should not lead to confusions (whereas mixing up different categories might do so).

\medskip

\noindent\textit{Throughout the following let $G$ be a group.}


\section{Preliminaries on graded rings}

First we recall our terminology for graded rings and coarsening functors.

\begin{no}\label{p10}
By a $G$-graded ring we mean a pair $(R,(R_g)_{g\in G})$ consisting of a ring $R$ and a family $(R_g)_{g\in G}$ of subgroups of the additive group of $R$ whose direct sum equals the additive group of $R$ such that $R_gR_h\subseteq R_{g+h}$ for $g,h\in G$. If no confusion can arise we denote a $G$-graded ring $(R,(R_g)_{g\in G})$ just by $R$. Accordingly, for a $G$-graded ring $R$ and $g\in G$ we denote by $R_g$ the component of degree $g$ of $R$. We set $R^{\hom}\dfgl\bigcup_{g\in G}R_g$ and call $\degsupp(R)\dfgl\{g\in G\mid R_g\neq 0\}$ the degree support of $R$. We say that $R$ has full support if $\degsupp(R)=G$ and that $R$ is trivially $G$-graded if $\degsupp(R)=\{0\}$. Given $G$-graded rings $R$ and $S$, by a morphism of $G$-graded rings from $R$ to $S$ we mean a morphism of rings $u\colon R\rightarrow S$ such that $u(R_g)\subseteq S_g$ for $g\in G$. By a $G$-graded $R$-algebra we mean a $G$-graded ring $S$ together with a morphism of $G$-graded rings $R\rightarrow S$. We denote by ${\sf GrAnn}^G$ the category of $G$-graded rings with this notion of morphism. This category has inductive and projective limits. In case $G=0$ we canonically identify ${\sf GrAnn}^G$ with the category of rings.
\end{no}

\begin{no}\label{p20}
Let $\psi\colon G\twoheadrightarrow H$ be an epimorphism of groups. For a $G$-graded ring $R$ we define an $H$-graded ring $R_{[\psi]}$, called the $\psi$-coarsening of $R$; its underlying ring is the ring underlying $R$, and its $H$-graduation is given by $(R_{[\psi]})_h=\bigoplus_{g\in\psi^{-1}(h)}R_g$ for $h\in H$. A morphism $u\colon R\rightarrow S$ of $G$-graded rings can be considered as a morphism of $H$-graded rings $R_{[\psi]}\rightarrow S_{[\psi]}$, and as such it is denoted by $u_{[\psi]}$. This gives rise to a functor $\bullet_{[\psi]}\colon{\sf GrAnn}^G\rightarrow{\sf GrAnn}^H$. This functor has a right adjoint, hence commutes with inductive limits, and it has a left adjoint if and only if $\ke(\psi)$ is finite (\cite[1.6; 1.8]{cih}). For a further epimorphism of groups $\varphi\colon H\twoheadrightarrow K$ we have $\bullet_{[\varphi\circ\psi]}=\bullet_{[\varphi]}\circ\bullet_{[\psi]}$.
\end{no}

\begin{no}\label{p25}
We denote by $\F_G$ the set of subgroups of finite type of $G$, ordered by inclusion, so that $G=\varinjlim_{U\in\F_G}U$.
\end{no}

\begin{no}\label{p30}
Let $F\subseteq G$ be a subgroup. For a $G$-graded ring $R$ we define an $F$-graded ring $R_{(F)}$ with underlying ring the subring $\bigoplus_{g\in F}R_g\subseteq R$ and with $F$-graduation $(R_g)_{g\in F}$. For an $F$-graded ring $S$ we define a $G$-graded ring $S^{(G)}$ with underlying ring the ring underlying $S$ and with $G$-graduation given by $S^{(G)}_g=S_g$ for $g\in F$ and $S^{(G)}_g=0$ for $g\in G\setminus F$. If $R$ is a $G$-graded ring and $\F$ is a set of subgroups of $G$, ordered by inclusion, whose inductive limit is $G$, then $R=\varinjlim_{F\in\F}((R_{(F)})^{(G)})$.
\end{no}

The next remark recalls the two different notions of graded group algebras and, more general, of graded monoid algebras.

\begin{no}\label{p40}
Let $M$ be a cancellable monoid, let $F$ be its group of differences, and let $R$ be a $G$-graded ring. The algebra of $M$ over $R$, furnished with its canonical $G\oplus F$-graduation, is denoted by $R[M]$ and called \textit{the finely graded algebra of $M$ over $R$,} and we denote by $(e_f)_{f\in F}$ its canonical basis. So, for $(g,f)\in G\oplus F$ we have $R[M]_{(g,f)}=R_ge_f$. Denoting by $\pi\colon G\oplus F\twoheadrightarrow G$ the canonical projection we set $R[M]_{[G]}\dfgl R[M]_{[\pi]}$ and call this \textit{the coarsely graded algebra of $M$ over $R$.} If $S$ is a $G$-graded $R$-algebra then $S[M]$ is a $G\oplus F$-graded $R[M]$-algebra, and $S[M]_{[G]}$ is a $G$-graded $R[M]_{[G]}$-algebra. We have $R[F]=\varinjlim_{U\in\F_F}R[U]^{(G\oplus F)}$ and $R[F]_{[G]}=\varinjlim_{U\in\F_F}R[U]_{[G]}$ (\ref{p25}).
\end{no}

We will need some facts about graded variants of simplicity (i.e., the property of ``being a field'') and entirety. Although they are probably well-known, we provide proofs for the readers convenience. Following Lang we use the term ``entire'' instead of ``integral'' (to avoid confusion with the notion of integrality over some ring which is central in this article) or ``domain'' (to avoid questions as whether a ``graded domain'' is the same as a ``domain furnished with a graduation''), and we use the term ``simple'' (which is more common in noncommutative algebra) instead of ``field'' for similar reasons.

\begin{no}\label{p50}
Let $R$ be a $G$-graded ring. We denote by $R^*$ the multiplicative group of invertible homogeneous elements of $R$ and by $\nzd(R)$ the multiplicative monoid of homogeneous non-zerodivisors of $R$. We call $R$ \textit{simple} if $R^*=R^{\hom}\setminus 0$, and \textit{entire} if $\nzd(R)=R^{\hom}\setminus 0$. If $R$ is entire then the $G$-graded ring of fractions $\nzd(R)^{-1}R$ is simple; we denote it by $Q(R)$ and call it \textit{the (graded) field of fractions of $R$.} If $\psi\colon G\twoheadrightarrow H$ is an epimorphism of groups and $R_{[\psi]}$ is simple or entire, then $R$ is so. Let $F\subseteq G$ be a subgroup. If $R$ is simple or entire, then $R_{(F)}$ is so, and an $F$-graded ring $S$ is simple or entire if and only if $S^{(G)}$ is so.
\end{no}

\begin{no}\label{p55}
Let $I$ be a nonempty right filtering preordered set, and let $((R_i)_{i\in I},(\varphi_{ij})_{i\leq j})$ be an inductive system in ${\sf GrAnn}^G$ over $I$. Analogously to \cite[I.10.3 Proposition 3]{a} we see that if $R_i$ is simple or entire for every $i\in I$, then $\varinjlim_{i\in I}R_i$ is simple or entire. If $R_i$ is entire for every $i\in I$ and $\varphi_{ij}$ is a monomorphism for all $i,j\in I$ with $i\leq j$, then by \cite[0.6.1.5]{ega} we get an inductive system $(Q(R_i))_{i\in I}$ in ${\sf GrAnn}^G$ over $I$ with $\varinjlim_{i\in I}Q(R_i)=Q(\varinjlim_{i\in I}R_i)$.
\end{no}

\begin{no}\label{p60}
Let $F\subseteq G$ be a subgroup and let $\leq$ be an ordering on $F$ that is compatible with its structure of group. The relation ``$g-h\in F_{\leq 0}$'' is the finest ordering on $G$ that is compatible with its structure of group and induces $\leq$ on $F$; we call it \textit{the canonical extension of $\leq$ to $G$.} If $\leq$ is a total ordering then its canonical extension to $G$ induces a total ordering on every equivalence class of $G$ modulo $F$.
\end{no}

\begin{lemma}\label{p70}
Let $\psi\colon G\twoheadrightarrow H$ be an epimorphism of groups such that $\ke(\psi)$ is torsionfree, let $R$ be an entire $G$-graded ring, and let $x,y\in R_{[\psi]}^{\hom}\setminus 0$ with $xy\in R^{\hom}$. Then, $x,y\in R^{\hom}$ and $xy\neq 0$.
\end{lemma}

\begin{proof}
(cf.~\cite[II.11.4 Proposition 8]{a}) By \cite[II.11.4 Lemme 1]{a} we can choose a total ordering on $\ke(\psi)$ that is compatible with its structure of group. Let $\leq$ denote its canonical extension to $G$ (\ref{p60}). Let $h\dfgl\deg(x)$ and $h'\dfgl\deg(y)$. There exist strictly increasing finite sequences $(g_i)_{i=0}^n$ in $\psi^{-1}(h)$ and $(g'_j)_{j=0}^m$ in $\psi^{-1}(h')$, $x_i\in R_{g_i}\setminus 0$ for $i\in[0,n]$, and $y_j\in R_{g'_j}\setminus 0$ for $j\in[0,m]$ such that $x=\sum_{i=0}^nx_i$ and $y=\sum_{j=0}^my_j$. If $k\in[0,n]$ and $l\in[0,m]$ with $g_k+g'_l=g_n+g'_m$ then $k=n$ and $l=m$ by \cite[VI.1.1 Proposition 1]{a}. This implies that the component of $xy$ of degree $g_n+g'_m$ equals $x_ny_m\neq 0$, so that $xy\neq 0$. As $x_0y_0\neq 0$ and $xy\in R^{\hom}$ it follows $g_0+g'_0=g_n+g'_m$, hence $n=m=0$ and therefore $x,y\in R^{\hom}$.
\end{proof}

\begin{cor}\label{p80}
Let $\psi\colon G\twoheadrightarrow H$ be an epimorphism of groups such that $\ke(\psi)$ is torsionfree, and let $R$ be an entire $G$-graded ring. Then, $R^*=R_{[\psi]}^*$.
\end{cor}

\begin{proof}
Clearly, $R^*\subseteq(R_{[\psi]})^*$. If $x\in(R_{[\psi]})^*\subseteq R_{[\psi]}^{\hom}\setminus 0$ then there exists $y\in R_{[\psi]}^{\hom}\setminus 0$ with $xy=1\in R^{\hom}$, so \ref{p70} implies $x\in R^{\hom}$, hence $x\in R^*$.
\end{proof}

\begin{no}\label{p100}
Let $R$ be a $G$-graded ring and let $F$ be a group. It is readily checked that $R[F]$ is simple or entire if and only if $R$ is so. Analogously to \cite[8.1]{gil} it is seen that $R[F]_{[G]}$ is entire if and only if $R$ is entire and $F$ is torsionfree. Together with \ref{p80} it follows that $R[F]_{[G]}$ is simple if and only if $R$ is simple and $F=0$.
\end{no}

\begin{prop}\label{p90}
Let $\psi\colon G\twoheadrightarrow H$ be an epimorphism of groups.

a) The following statements are equivalent: (i) $\psi$ is an isomorphism; (ii) $\psi$-coarsening preserves simplicity.

b) The following statements are equivalent: (i) $\ke(\psi)$ is torsionfree; (ii) $\psi$-coarsening preserves entirety; (iii) $\psi$-coarsening maps simple $G$-graded rings to entire $H$-graded rings.\footnote{In case $H=0$ the implication (i)$\Rightarrow$(ii) is \cite[II.11.4 Proposition 8]{a}.}
\end{prop}

\begin{proof}
If $K$ is a field and $R=K[\ke(\psi)]^{(G)}$, then $R$ is simple and $R_{[\psi]}$ is trivially $H$-graded, hence $R_{[\psi]}$ is simple or entire if and only if $R_{[0]}$ is so (\ref{p100}, \ref{p50}). If $\ke(\psi)\neq 0$ then $R_{[0]}$ is not simple, and if $\ke(\psi)$ is not torsionfree then $R_{[0]}$ is not entire (\ref{p100}). This proves a) and the implication (iii)$\Rightarrow$(i) in b). The remaining claims follow from \ref{p70}.
\end{proof}

\begin{no}\label{p91}
A $G$-graded ring $R$ is called \textit{reduced} if $0$ is its only nilpotent homogeneous element. With arguments similar to those above one can show that statements (i)--(iii) in \ref{p90}~b) are also equivalent to the following: (iv) $\psi$-coarsening preserves reducedness; (v) $\psi$-coarsening maps simple $G$-graded rings to reduced $H$-graded rings. We will make no use of this fact.
\end{no}

Finally we make some remarks on a graded variant of noetherianness.

\begin{no}\label{p110}
Let $R$ be a $G$-graded ring. We call $R$ \textit{noetherian} if ascending sequences of $G$-graded ideals of $R$ are stationary, or -- equivalently -- if every $G$-graded ideal of $R$ is of finite type. Analogously to the ungraded case one can prove a graded version of Hilbert's Basissatz: If $R$ is noetherian then so are $G$-graded $R$-algebras of finite type. If $\psi\colon G\twoheadrightarrow H$ is an epimorphism of groups and $R_{[\psi]}$ is noetherian, then $R$ is noetherian. Let $F\subseteq G$ be a subgroup. It follows from \cite[2.1]{gy} that if $R$ is noetherian then so is $R_{(F)}$. Moreover, an $F$-graded ring $S$ is noetherian if and only if $S^{(G)}$ is so.

If $F$ is a group then it follows from \cite[2.1]{gy} and the fact that $e_f\in R[F]^*$ for $f\in F$ that $R[F]$ is noetherian if and only if $R$ is so. Analogously to \cite[7.7]{gil} one sees that $R[F]_{[G]}$ is noetherian if and only if $R$ is noetherian and $F$ is of finite type. More general, it follows readily from a result by Goto and Yamagishi (\cite[1.1]{gy}) that $G$ is of finite type if and only if $\psi$-coarsening preserves noetherianness for every epimorphism of groups $\psi\colon G\twoheadrightarrow H$. (This was proven again two years later by N\v{a}st\v{a}sescu and Van Oystaeyen (\cite[2.1]{nv}).)
\end{no}


\section{Relative integral closures}

We begin this section with basic definitions and first properties of relative (complete) integral closures.

\begin{no}\label{a10}
Let $R$ be a $G$-graded ring and let $S$ be a $G$-graded $R$-algebra. An element $x\in S^{\hom}$ is called \textit{integral over $R$} if it is a zero of a monic polynomial in one indeterminate with coefficients in $R^{\hom}$. This is the case if and only if $x$, considered as an element of $S_{[0]}$, is integral over $R_{[0]}$, as is seen analogously to the first paragraph of \cite[V.1.8]{ac}. Hence, using \cite[V.1.1 Th\'eor\`eme 1]{ac} we see that for $x\in S^{\hom}$ the following statements are equivalent: (i) $x$ is integral over $R$; (ii) the $G$-graded $R$-module underlying the $G$-graded $R$-algebra $R[x]$ is of finite type; (iii) there exists a $G$-graded sub-$R$-algebra of $S$ containing $R[x]$ whose underlying $G$-graded $R$-module is of finite type.

An element $x\in S^{\hom}$ is called \textit{almost integral over $R$} if there exists a $G$-graded sub-$R$-module $T\subseteq S$ of finite type containing $R[x]$. This is the case if and only if $x$, considered as an element of $S_{[0]}$, is almost integral over $R_{[0]}$. Indeed, this condition is obviously necessary. It is also sufficient, for if $T\subseteq S_{[0]}$ is a sub-$R_{[0]}$-module of finite type containing $R_{[0]}[x]$ then the $G$-graded sub-$R$-module $T'\subseteq S$ generated by the set of homogeneous components of elements of $T$ is of finite type and contains $T$, hence $R[x]$. It follows from the first paragraph that if $x\in S^{\hom}$ is integral over $R$ then it is almost integral over $R$; analogously to \cite[V.1.1 Proposition 1 Corollaire]{ac} it is seen that the converse is true if $R$ is noetherian (\ref{p110}).
\end{no}

\begin{no}\label{a20}
Let $R$ be a $G$-graded ring and let $S$ be a $G$-graded $R$-algebra. The $G$-graded sub-$R$-algebra of $S$ generated by the set of elements of $S^{\hom}$ that are (almost) integral over $R$ is denoted by $\inte(R,S)$ (or $\cinte(R,S)$, resp.) and is called \textit{the (complete) integral closure of $R$ in $S$.} We have $\inte(R,S)\subseteq\cinte(R,S)$, with equality if $R$ is noetherian (\ref{a10}). For an epimorphism of groups $\psi\colon G\twoheadrightarrow H$ we have $\inte(R,S)_{[\psi]}\subseteq\inte(R_{[\psi]},S_{[\psi]})$ and $\cinte(R,S)_{[\psi]}\subseteq\cinte(R_{[\psi]},S_{[\psi]})$ (\ref{a10}).

Let $R'$ denote the image of $R$ in $S$. We say that $R$ is \textit{(completely) integrally closed in $S$} if $R'=\inte(R,S)$ (or $R'=\cinte(R,S)$, resp.), and that $S$ is \textit{(almost) integral over $R$} if $\inte(R,S)=S$ (or $\cinte(R,S)=S$, resp.). If $R$ is completely integrally closed in $S$ then it is integrally closed in $S$, and if $S$ is integral over $R$ then it is almost integral over $R$; the converse statements are true if $R$ is noetherian. If $\psi\colon G\twoheadrightarrow H$ is an epimorphism of groups, then $S$ is (almost) integral over $R$ if and only if $S_{[\psi]}$ is (almost) integral over $R_{[\psi]}$, and if $R_{[\psi]}$ is (completely) integrally closed in $S_{[\psi]}$ then $R$ is (completely) integrally closed in $S$. If $G\subseteq F$ is a subgroup then $\inte(R,S)^{(F)}=\inte(R^{(F)},S^{(F)})$ and $\cinte(R,S)^{(F)}=\cinte(R^{(F)},S^{(F)})$, hence $R$ is (completely) integrally closed in $S$ if and only if $R^{(F)}$ is (completely) integrally closed in $S^{(F)}$.

From \cite[V.1.1 Proposition 4 Corollaire 1]{ac} and \cite[\S 135, p.~180]{vdw}\footnote{Note that van der Waerden calls ``integral'' what we call ``almost integral''.} we know that sums and products of elements of $S_{[0]}$ that are (almost) integral over $R_{[0]}$ are again (almost) integral over $R_{[0]}$. Hence, $\inte(R,S)^{\hom}$ (or $\cinte(R,S)^{\hom}$, resp.) equals the set of homogeneous elements of $S$ that are (almost) integral over $R$, and thus $\inte(R,S)$ (or $\cinte(R,S)$, resp.) is (almost) integral over $R$ by the above. Moreover, $\inte(R,S)$ is integrally closed in $S$ by \cite[V.1.2 Proposition 7]{ac}. One should note that $\cinte(R,S)$ is not necessarily completely integrally closed in $S$, not even if $R$ is entire and $S=Q(R)$ (\cite[Example 1]{gilhei}).
\end{no}

\begin{no}\label{a11}
Suppose we have a commutative diagram of $G$-graded rings $$\xymatrix@R15pt@C40pt{R\ar[r]\ar[d]&S\ar[d]^h\\\overline{R}\ar[r]&\overline{S}.}$$ If $x\in S^{\hom}$ is (almost) integral over $R$, then $h(x)\in\overline{S}^{\hom}$ is (almost) integral over $\overline{R}$ (\ref{a10}, \cite[V.1.1 Proposition 2]{ac}, \cite[13.5]{mit}). Hence, if the diagram above is cartesian and $\overline{R}$ is (completely) integrally closed in $\overline{S}$, then $R$ is (completely) integrally closed in $S$.
\end{no}

\begin{no}\label{a30}
Let $R$ be a $G$-graded ring, let $S$ be a $G$-graded $R$-algebra, and let $T\subseteq R^{\hom}$ be a subset. Analogously to \cite[V.1.5 Proposition 16]{ac} one shows that $T^{-1}\inte(R,S)=\inte(T^{-1}R,T^{-1}S)$. Hence, if $R$ is integrally closed in $S$ then $T^{-1}R$ is integrally closed in $T^{-1}S$.

Note that there is no analogous statement for complete integral closures. Although $T^{-1}\cinte(R,S)\subseteq\cinte(T^{-1}R,T^{-1}S)$ by \ref{a11}, this need not be an equality. In fact, by \cite[V.1 Exercice 12]{ac} there exists an entire ring $R$ that is completely integrally closed in $Q(R)$ and a subset $T\subseteq R\setminus 0$ such that $Q(R)$ is the complete integral closure of $T^{-1}R$.
\end{no}

\begin{no}\label{a50}
Let $I$ be a right filtering preordered set, and let $(u_i)_{i\in I}\colon(R_i)_{i\in I}\rightarrow(S_i)_{i\in I}$ be a morphism of inductive systems in ${\sf GrAnn}^G$ over $I$. By \ref{a11} we have inductive systems $(\inte(R_i,S_i))_{i\in I}$ and $(\cinte(R_i,S_i))_{i\in I}$ in ${\sf GrAnn}^G$ over $I$, and we can consider the sub-$\varinjlim_{i\in I}R_i$-algebras $$\textstyle\varinjlim_{i\in I}\inte(R_i,S_i)\subseteq\varinjlim_{i\in I}\cinte(R_i,S_i)\subseteq\varinjlim_{i\in I}S_i$$ and compare them with the sub-$\varinjlim_{i\in I}R_i$-algebras $$\textstyle\inte(\varinjlim_{i\in I}R_i,\varinjlim_{i\in I}S_i)\subseteq\cinte(\varinjlim_{i\in I}R_i,\varinjlim_{i\in I}S_i)\subseteq\varinjlim_{i\in I}S_i.$$ Analogously to \cite[0.6.5.12]{ega} it follows $\varinjlim_{i\in I}\inte(R_i,S_i)=\inte(\varinjlim_{i\in I}R_i,\varinjlim_{i\in I}S_i)$, hence if $R_i$ is integrally closed in $S_i$ for every $i\in I$ then $\varinjlim_{i\in I}R_i$ is integrally closed in $\varinjlim_{i\in I}S_i$.

Note that there is no analogous statement for complete integral closures. Although  $\varinjlim_{i\in I}\cinte(R_i,S_i)\subseteq\cinte(\varinjlim_{i\in I}R_i,\varinjlim_{i\in I}S_i)$ by \ref{a11}, this need not be an equality (but cf.~\ref{lem3}). In fact, by \cite[V.1 Exercice 11 b)]{ac} there exist a field $K$ and an increasing family $(R_n)_{n\in\N}$ of subrings of $K$ such that $R_n$ is completely integrally closed in $Q(R_n)=K$ for every $n\in\N$ and that $\varinjlim_{n\in\N}R_n$ is not completely integrally closed in $Q(\varinjlim_{n\in\N}R_n)=K$.
\end{no}

Now we turn to the behaviour of finely and coarsely graded group algebras with respect to relative (complete) integral closures.

\begin{thm}\label{a80}
Let $R$ be a $G$-graded ring.

a) Formation of finely graded group algebras over $R$ commutes with relative (complete) integral closure.

b) Let $S$ be a $G$-graded $R$-algebra, and let $F$ be a group. Then, $R$ is (completely) integrally closed in $S$ if and only if $R[F]$ is (completely) integrally closed in $S[F]$.
\end{thm}

\begin{proof}
a) Let $F$ be a group, and let $S$ be a $G$-graded $R$-algebra. Let $x\in S[F]^{\hom}$. There are $s\in S^{\hom}$ and $f\in F$ with $x=se_f$. If $x\in\inte(R,S)[F]^{\hom}$ then $s\in\inte(R,S)^{\hom}$, hence $s\in\inte(R[F],S[F])$ (\ref{a11}), and as $e_f\in\inte(R[F],S[F])$ it follows $x=se_f\in\inte(R[F],S[F])$. This shows $\inte(R,S)[F]\subseteq\inte(R[F],S[F])$. Conversely, suppose that $x\in\inte(R[F],S[F])^{\hom}$. As $e_f\in R[F]^*$ it follows $s\in\inte(R[F],S[F])^{\hom}$. So, there is a finite subset $E\subseteq S[F]^{\hom}$ such that the $G\oplus F$-graded sub-$R[F]$-algebra of $S[F]$ generated by $E$ contains $R[F][s]$. As $e_h\in R[F]^*$ for every $h\in F$ we can suppose $E\subseteq S$. If $n\in\N$ then $s^n$ is an $R[F]$-linear combination of products in $E$, and comparing the coefficients of $e_0$ shows that $s^n$ is an $R$-linear combination of products in $E$. Thus, $R[s]$ is contained in the $G$-graded sub-$R$-algebra of $S$ generated by $E$, hence $s\in\inte(R,S)$, and therefore $x\in\inte(R,S)[F]$. This shows $\inte(R[F],S[F])\subseteq\inte(R,S)[F]$. The claim for complete integral closures follows analogously. b) follows immediately from a).
\end{proof}

\begin{thm}\label{a90}
Let $F$ be a group. The following statements are equivalent:\footnote{In case $G=0$ the implication (iii)$\Rightarrow$(i) is \cite[V.1 Exercice 24]{ac}.}
\begin{equi}
\item[(i)] Formation of coarsely graded algebras of $F$ over $G$-graded rings commutes with relative integral closure;
\item[(i')] Formation of coarsely graded algebras of $F$ over $G$-graded rings commutes with relative complete integral closure;
\item[(ii)] If $R$ is a $G$-graded ring and $S$ is a $G$-graded $R$-algebra, then $R$ is integrally closed in $S$ if and only if $R[F]_{[G]}$ is integrally closed in $S[F]_{[G]}$;
\item[(ii')] If $R$ is a $G$-graded ring and $S$ is a $G$-graded $R$-algebra, then $R$ is completely integrally closed in $S$ if and only if $R[F]_{[G]}$ is completely integrally closed in $S[F]_{[G]}$;
\item[(iii)] $F$ is torsionfree.
\end{equi}
\end{thm}

\begin{proof}
``(i)$\Rightarrow$(ii)'' and ``(i')$\Rightarrow$(ii')'': Immediately from \ref{a11}.

``(ii)$\Rightarrow$(iii)'' and ``(ii')$\Rightarrow$(iii)'': Suppose that $F$ is not torsionfree. It suffices to find a noetherian ring $R$ and an $R$-algebra $S$ such that $R$ is integrally closed in $S$ and that $R[F]_{[0]}$ is not integrally closed in $S[F]_{[0]}$, for then furnishing $R$ and $S$ with trivial $G$-graduations it follows that $R[F]_{[G]}$ is not integrally closed in $S[F]_{[G]}$ (\ref{a20}). The ring $\Z$ is noetherian and integrally closed in $\Q$. By hypothesis there exist $g\in F\setminus 0$ and $n\in\N_{>1}$ with $ng=0$, so that $e_g^n=1\in\Q[F]_{[0]}$. It is readily checked that $f\dfgl\sum_{i=0}^{n-1}\frac{1}{n}e_g^i\in\Q[F]_{[0]}\setminus\Z[F]_{[0]}$ is idempotent. Setting $c\dfgl 1+(n-1)e_g^{n-1}\in\Z[F]_{[0]}$ we get $$\textstyle d\dfgl fc=f+(n-1)fe_g^{n-1}=f+\sum_{i=0}^{n-1}\frac{n-1}{n}e_g^{i+n-1}=f+\sum_{i=0}^{n-1}\frac{n-1}{n}e_g^{i-1}=$$$$\textstyle f+\frac{n-1}{n}e_g^{-1}+\sum_{i=0}^{n-2}\frac{n-1}{n}e_g^i+\frac{n-1}{n}e_g^{n-1}-\frac{n-1}{n}e_g^{n-1}=f+(n-1)f=nf\in\Z[F]_{[0]}.$$ Therefore, $f^2+(c-1)f-d=f+d-f-d=0$ yields an integral equation for $f$ over $\Z[F]_{[0]}$. Thus, $\Z[F]_{[0]}$ is not integrally closed in $\Q[F]_{[0]}$.

``(iii)$\Rightarrow$(i)'' and ``(iii)$\Rightarrow$(i')'': Without loss of generality suppose $G=0$ (\ref{a20}). Suppose that $F$ is torsionfree, let $R$ be a ring, and let $S$ be an $R$-algebra. If $n\in\N$ then $\inte(R,S)[\N^n]_{[0]}=\inte(R[\N^n]_{[0]},S[\N^n]_{[0]})$ (\cite[V.1.3 Proposition 12]{ac}), hence $$\inte(R[\Z^n]_{[0]},S[\Z^n]_{[0]})=\inte(R,S)[\Z^n]_{[0]}$$ (\ref{a30}). This proves (i) in case $F$ is of finite type, and so we get (i) in general by \ref{p40}, \ref{a20} and \ref{a50}. It remains to show (i'). The inclusion $\cinte(R,S)[F]_{[0]}\subseteq\cinte(R[F]_{[0]},S[F]_{[0]})$ follows immediately from \ref{a20} and \ref{a11}. We prove the converse inclusion analogously to \cite[Proposition 1]{gilhei}. Since $F$ is torsionfree we can choose a total ordering $\leq$ on $F$ that is compatible with its structure of group (\cite[II.11.4 Lemme 1]{a}). Let $x\in\cinte(S[F]_{[0]},R[F]_{[0]})$. There are $n\in\N$, a family $(x_i)_{i=1}^n$ in $S\setminus 0$, and a strictly increasing family $(f_i)_{i=1}^n$ in $F$ with $x=\sum_{i=1}^nx_ie_{f_i}$. We prove by induction on $n$ that $x\in\cinte(R,S)[F]_{[0]}$. If $n=0$ this is clear. Suppose that $n>0$ and that the claim is true for strictly smaller values of $n$. There exists a finite subset $P\subseteq S[F]_{[0]}$ with $R[F]_{[0]}[x]\subseteq\langle P\rangle_{R[F]}$. Let $Q$ denote the finite set of coefficients of elements of $P$. Let $k\in\N$. There exists a family $(s_p)_{p\in P}$ in $R[F]$ with $x^k=\sum_{p\in P}s_pp$. By means of the ordering of $F$ we see that $x_n^k$ is the coefficient of $e_{f_{kn}}$ in $x^k$, hence the coefficient of $e_{f_{kn}}$ in $\sum_{p\in P}s_pp$. This latter being an $R$-linear combination of $Q$ we get $x_n^k\in\langle Q\rangle_R$. It follows $R[x_n]\subseteq\langle Q\rangle_R$, and thus $x_n\in\cinte(R,S)$. So, we get $x_ne_{f_n}\in\cinte(R[F]_{[0]},S[F]_{[0]})$ (\ref{a20}, \ref{a11}), hence $x-x_ne_{f_n}\in\cinte(R[F]_{[0]},S[F]_{[0]})$, thus $x-x_ne_{f_n}\in\cinte(R,S)[F]_{[0]}$ by our hypothesis, and therefore $x\in\cinte(R,S)[F]_{[0]}$ as desired.
\end{proof}

\begin{no}\label{a91}
The proof above shows that \ref{a90} remains true if we replace ``If $R$ is a $G$-graded ring'' by ``If $R$ is a noetherian $G$-graded ring'' in (ii) and (ii').
\end{no}

The rest of this section is devoted to the study of the behaviour of relative (complete) integral closures under arbitrary coarsening functors. Although we are not able to characterise those coarsenings with good behaviour, we identify two properties of the coarsening, one that implies good behaviour of (complete) integral closures, and one that is implied by good behaviour of (complete) integral closures.

\begin{no}\label{a100}
Let $\psi\colon G\twoheadrightarrow H$ be an epimorphism of groups. We say that \textit{$\psi$-coarsening commutes with relative (complete) integral closure} if $\inte(R,S)_{[\psi]}=\inte(R_{[\psi]},S_{[\psi]})$ (or $\cinte(R,S)_{[\psi]}=\cinte(R_{[\psi]},S_{[\psi]})$, resp.) for every $G$-graded ring $R$ and every $G$-graded $R$-algebra $S$.
\end{no}

\begin{prop}\label{a101}
Let $\psi\colon G\twoheadrightarrow H$ be an epimorphism of groups. We consider the following statements:
\begin{aufz}
\item[(1)] $\psi$-coarsening commutes with relative integral closure;
\item[(1')] $\psi$-coarsening commutes with relative complete integral closure;
\item[(2)] If $R$ is a $G$-graded ring, $S$ is a $G$-graded $R$-algebra, and $x\in S_{[\psi]}^{\hom}$, then $x$ is integral over $R_{[\psi]}$ if and only if all its homogeneous components (with respect to the $G$-graduation) are integral over $R$;
\item[(2')] If $R$ is a $G$-graded ring, $S$ is a $G$-graded $R$-algebra, and $x\in S_{[\psi]}^{\hom}$, then $x$ is almost integral over $R_{[\psi]}$ if and only if all its homogeneous components (with respect to the $G$-graduation) are almost integral over $R$;
\item[(3)] If $R$ is a $G$-graded ring and $S$ is a $G$-graded $R$-algebra, then $R$ is integrally closed in $S$ if and only if $R_{[\psi]}$ is integrally closed in $S_{[\psi]}$.
\item[(3')] If $R$ is a $G$-graded ring and $S$ is a $G$-graded $R$-algebra, then $R$ is completely integrally closed in $S$ if and only if $R_{[\psi]}$ is completely integrally closed in $S_{[\psi]}$.
\end{aufz}
Then, we have (1)$\Leftrightarrow$(2)$\Leftrightarrow$(3) and (1')$\Leftrightarrow$(2')$\Rightarrow$(3').
\end{prop}

\begin{proof}
The implications ``(1)$\Leftrightarrow$(2)$\Rightarrow$(3)'' and ``(1')$\Leftrightarrow$(2')$\Rightarrow$(3')'' follow immediately from the definitions. Suppose (3) is true, let $R$ be a $G$-graded ring $R$, and let $S$ be a $G$-graded $R$-algebra. As $\inte(R,S)$ is integrally closed in $S$ (\ref{a20}) it follows that $\inte(R,S)_{[\psi]}$ is integrally closed in $S_{[\psi]}$, implying $$\inte(R_{[\psi]},S_{[\psi]})\subseteq\inte(\inte(R,S)_{[\psi]},S_{[\psi]})=\inte(R,S)_{[\psi]}\subseteq\inte(R_{[\psi]},S_{[\psi]})$$ (\ref{a20}) and thus the claim.
\end{proof}

The argument above showing that (3) implies (1) cannot be used to show that (3') implies (1'), as $\cinte(R,S)$ is not necessarily completely integrally closed in $S$ (\ref{a20}).

\begin{no}\label{a110}
Let $\psi\colon G\twoheadrightarrow H$ be an epimorphism of groups, suppose that there exists a section $\pi\colon H\rightarrow G$ of $\psi$ in the category of groups, and let $R$ be a $G$-graded ring. For $g\in G$ there is a morphism of groups $$j_{R,g}^{\pi}\colon R_g\rightarrow R_{[0]}[\ke(\psi)],x\mapsto xe_{g+\pi(\psi(g))}.$$ The family $(j^{\pi}_{R,g})_{g\in G}$ induces a morphism of groups $j^{\pi}_R\colon \bigoplus_{g\in G}R_g\rightarrow R_{[0]}[\ke(\psi)]_{[0]}$ that is readily checked to be a morphism $j^{\pi}_R\colon R_{[\psi]}\rightarrow R_{[\psi]}[\ke(\psi)]_{[H]}$ of $H$-graded rings.
\end{no}

\begin{thm}\label{a120}
Let $\psi\colon G\twoheadrightarrow H$ be an epimorphism of groups.

a) If $\ke(\psi)$ is contained in a torsionfree direct summand of $G$ then $\psi$-coarsening commutes with relative (complete) integral closure.

b) If $\psi$-coarsening commutes with relative (complete) integral closure then $\ke(\psi)$ is torsionfree.
\end{thm}

\begin{proof}
a) First, we consider the case that $\ke(\psi)$ itself is a torsionfree direct summand of $G$. Let $R$ be a $G$-graded ring, let $S$ be a $G$-graded $R$-algebra, and let $x\in S^{\hom}_{[\psi]}$ be (almost) integral over $R_{[\psi]}$. As $\ke(\psi)$ is a direct summand of $G$ there exists a section $\pi$ of $\psi$ in the category of groups. So, we have a commutative diagram $$\xymatrix{R_{[\psi]}\ar[d]_{j^{\pi}_R}\ar[r]&S_{[\psi]}\ar[d]^{j^{\pi}_S}\\R_{[\psi]}[\ke(\psi)]_{[H]}\ar[r]&S_{[\psi]}[\ke(\psi)]_{[H]}}$$ of $H$-graded rings (\ref{a110}). Since $\ke(\psi)$ is torsionfree it follows $$j^{\pi}_S(x)\in\inte(R_{[\psi]}[\ke(\psi)]_{[H]},S_{[\psi]}[\ke(\psi)]_{[H]})=\inte(R_{[\psi]},S_{[\psi]})[\ke(\psi)]_{[H]}$$ (and similarly for complete integral closures) by \ref{a11} and \ref{a90}. By the construction of $j^{\pi}_S$ this implies $x_g\in\inte(R_{[\psi]},S_{[\psi]})$ (or $x_g\in\cinte(R_{[\psi]},S_{[\psi]})$, resp.) for every $g\in G$, and thus the claim (\ref{a101}).

Next, we consider the general case. Let $F$ be a torsionfree direct summand of $G$ containing $\ke(\psi)$, let $\chi\colon G\twoheadrightarrow G/F$ be the canonical projection and let $\lambda\colon H\twoheadrightarrow G/F$ be the induced epimorphism of groups, so that $\lambda\circ\psi=\chi$. Let $R$ be a $G$-graded ring, and let $S$ be a $G$-graded $R$-algebra. By \ref{a100} and the first paragraph, $$\inte(R_{[\chi]},S_{[\chi]})=\inte(R,S)_{[\chi]}=(\inte(R,S)_{[\psi]})_{[\lambda]}\subseteq$$$$\inte(R_{[\psi]},S_{[\psi]})_{[\lambda]}\subseteq\inte((R_{[\psi]})_{[\lambda]},(S_{[\psi]})_{[\lambda]})=\inte(R_{[\chi]},S_{[\chi]}),$$ hence $(\inte(R,S)_{[\psi]})_{[\lambda]}=\inte(R_{[\psi]},S_{[\psi]})_{[\lambda]}$ and therefore $\inte(R,S)_{[\psi]}=\inte(R_{[\psi]},S_{[\psi]})$ (or the analogous statement for complete integral closures) as desired.

b) Suppose $K\dfgl\ke(\psi)$ is not torsionfree. By \ref{a90} and \ref{a91} there exist a noetherian ring $R$ and an $R$-algebra $S$ such that $R$ is integrally closed in $S$ (hence completely integrally closed in $S$) and that $R[K]_{[0]}$ is not integrally closed in $S[K]_{[0]}$ (hence not completely integrally closed in $S[K]_{[0]}$ (\ref{a20})). Then, $R[K]$ is completely integrally closed in $S[K]$ (\ref{a80}). Extending the $K$-graduations of $R$ and $S$ trivially to $G$-graduations it follows that $R[K]_{[G]}$ is completely integrally closed in $S[K]_{[G]}$, while $(R[K]_{[G]})_{[\psi]}$ is not integrally closed in $(S[K]_{[G]})_{[\psi]}$. This proves the claim.
\end{proof}

\begin{cor}\label{a130}
Let $\psi\colon G\twoheadrightarrow H$ be an epimorphism of groups. If $G$ is torsionfree then $\psi$-coarsening commutes with relative (complete) integral closure.\footnote{In case $H=0$ the statement about relative integral closures is \cite[V.1 Exercice 25]{ac}.}
\end{cor}

\begin{proof}
Immediately from \ref{a120}.
\end{proof}

\begin{no}\label{a140}
Supposing that the torsion subgroup $T$ of $G$ is a direct summand of $G$, it is readily checked that a subgroup $F\subseteq G$ is contained in a torsionfree direct summand of $G$ if and only if the composition of canonical morphisms $T\hookrightarrow G\twoheadrightarrow G/F$ has a retraction.

A torsionfree subgroup $F\subseteq G$ is not necessarily contained in a torsionfree direct summand of $G$, not even if $G$ is of finite type. Using the criterion above one checks that a counterexample is provided by $G=\Z\oplus\Z/n\Z$ for $n\in\N_{>1}$ and $F=\langle(n,\overline{1})\rangle_{\Z}\subseteq G$.
\end{no}

\begin{qu}
Let $\psi\colon G\twoheadrightarrow H$ be an epimorphism of groups. The above result gives rise to the following questions:
\begin{aufz}
\item[a)] \textit{If $\ke(\psi)$ is torsionfree, does $\psi$-coarsening commute with (complete) integral closure?}
\item[b)] \textit{If $\psi$-coarsening commutes with (complete) integral closure, is then $\ke(\psi)$ contained in a torsionfree direct summand of $G$?}
\end{aufz}
Note that, by \ref{a140}, at most one of these questions has a positive answer.
\end{qu}


\section{Integral closures of entire graded rings}

In this section we consider (complete) integral closures of entire graded rings in their graded fields of fractions. We start with the relevant definitions and basic properties.

\begin{no}\label{4.400}
Let $R$ be an entire $G$-graded ring. The (complete) integral closure of $R$ in $Q(R)$ is denoted by $\inte(R)$ (or $\cinte(R)$, resp.) and is called \textit{the (complete) integral closure of $R$.} We say that $R$ is \textit{(completely) integrally closed} if it is (completely) integrally closed in $Q(R)$. Keep in mind that $\inte(R)$ is integrally closed, but that $\cinte(R)$ is not necessarily completely integrally closed (\ref{a20}). If $\psi\colon G\twoheadrightarrow H$ is an epimorphism of groups and $R$ is a $G$-graded ring such that $R_{[\psi]}$ (and hence $R$) is entire, then $Q(R)_{[\psi]}$ is entire and $$R_{[\psi]}\subseteq Q(R)_{[\psi]}\subseteq Q(Q(R)_{[\psi]})=Q(R_{[\psi]}).$$ From \ref{a100} it follows $\inte(R)_{[\psi]}\subseteq\inte(R_{[\psi]})$ and $\cinte(R)_{[\psi]}\subseteq\cinte(R_{[\psi]})$. Hence, if $R_{[\psi]}$ is (completely) integrally closed then so is $R$. We say that \textit{$\psi$-coarsening commutes with (complete) integral closure} if whenever $R$ is an entire $G$-graded ring then $R_{[\psi]}$ is entire and $\inte(R)_{[\psi]}=\inte(R_{[\psi]})$ (or $\cinte(R)_{[\psi]}=\cinte(R_{[\psi]})$, resp.). Clearly, if $\psi$-coarsening commutes with (complete) integral closure then $\ke(\psi)$ is torsionfree (\ref{p90}~b)). Let $F\subseteq G$ be a subgroup. An $F$-graded ring $S$ is entire and (completely) integrally closed if and only if $S^{(G)}$ is so.
\end{no}

\begin{no}\label{lem3}
Let $I$ be a nonempty right filtering preordered set, and let $((R_i)_{i\in I},(\varphi_{ij})_{i\leq j})$ be an inductive system in ${\sf GrAnn}^G$ over $I$ such that $R_i$ is entire for every $i\in I$ and that $\varphi_{ij}$ is a monomorphism for all $i,j\in I$ with $i\leq j$. By \ref{a50} and \ref{p55} we have inductive systems $(\inte(R_i))_{i\in I}$ and $(\cinte(R_i))_{i\in I}$ in ${\sf GrAnn}^G$ over $I$, and we can consider the sub-$\varinjlim_{i\in I}R_i$-algebras $$\varinjlim_{i\in I}\inte(R_i)\subseteq\varinjlim_{i\in I}\cinte(R_i)\subseteq Q(\varinjlim_{i\in I}R_i)$$ and compare them with the sub-$\varinjlim_{i\in I}R_i$-algebras $$\inte(\varinjlim_{i\in I}R_i)\subseteq\cinte(\varinjlim_{i\in I}R_i)\subseteq Q(\varinjlim_{i\in I}R_i).$$ It follows immediately from \ref{a50} and \ref{p55} that $\varinjlim_{i\in I}\inte(R_i)=\inte(\varinjlim_{i\in I}R_i)$. Hence, if $R_i$ is integrally closed for every $i\in I$ then $\varinjlim_{i\in I}R_i$ is integrally closed, and $\varinjlim_{i\in I}\cinte(R_i)\subseteq\cinte(\varinjlim_{i\in I}R_i)$.

Suppose now in addition that $(\varinjlim_{i\in I}R_i)\cap Q(R_i)=R_i$ for every $i\in I$. Then, analogously to \cite[2.1]{and} one sees that $\varinjlim_{i\in I}\cinte(R_i)=\cinte(\varinjlim_{i\in I}R_i)$, hence if $R_i$ is completely integrally closed for every $i\in I$ then so is $\varinjlim_{i\in I}R_i$. This additional hypothesis is fulfilled in case $R$ is a $G$-graded ring, $F$ is a group, $I=\F_F$ (\ref{p25}), and $R_U$ equals $R[U]^{(G\oplus F)}$ or $R[U]_{[G]}$ for $U\in\F_F$.
\end{no}

\begin{no}\label{trans10}
Let $R$, $S$ and $T$ be $G$-graded rings such that $R\subseteq S\subseteq T$ as graded subrings. Clearly, $\cinte(R,S)\subseteq\cinte(R,T)\cap S$. Gilmer and Heinzer found an (ungraded) example showing that this is not necessarily an equality (\cite[Example 2]{gilhei}), not even if $R$, $S$ and $T$ are entire and have the same field of fractions. In \cite[Proposition 2]{gilhei} they also presented the following criterion for this inclusion to be an equality, whose graded variant is proven analogously: If for every $G$-graded sub-$S$-module of finite type $M\subseteq T$ with $R\subseteq M$ there exists a $G$-graded $S$-module $N$ containing $M$ such that $R$ is a direct summand of $N$, then $\cinte(R,S)=\cinte(R,T)\cap S$.

In \cite[Remark 2]{gilhei} Gilmer and Heinzer claim (again in the ungraded case) that this criterion applies if $S$ is principal. As this would be helpful to us later (\ref{f10}, \ref{4.600}) we take the opportunity to point out that it is wrong. Namely, suppose that $S$ is not simple, that $T=Q(S)$, and that the hypothesis of the criterion is fulfilled. Let $x\in S\setminus 0$. There exists an $S$-module $N$ containing $\langle x^{-1}\rangle_S\subseteq T$ such that $S$ is a direct summand of $N$. The tensor product with $S/\langle x\rangle_S$ of the canonical injection $S\hookrightarrow N$ has a retraction, but it also factors over the zero morphism $S/\langle x\rangle_S\rightarrow\langle x^{-1}\rangle_S/S$. This implies $x\in S^*$, yielding the contradiction that $S$ is simple.
\end{no}

Now we consider graded group algebras. We will show that both variants behave well with integral closures and that the finely graded variant behaves also well with complete integral closure.

\begin{thm}\label{fein}
a) Formation of finely graded group algebras over entire $G$-graded rings commutes with (complete) integral closure.

b) If $F$ is a group, then an entire $G$-graded ring $R$ is (completely) integrally closed if and only if $R[F]$ is so.
\end{thm}

\begin{proof}
Keeping in mind that $Q(R)[F]=Q(R[F])$ (\ref{p100}) this follows immediately from \ref{a80}.
\end{proof}

\begin{lemma}\label{f20}
If $R$ is a simple $G$-graded ring then $R[\Z]_{[G]}$ is entire and completely integrally closed.
\end{lemma}

\begin{proof}
First we note that $S\dfgl R[\Z]_{[G]}$ is entire (\ref{p100}). The argument in \cite[IV.1.6 Proposition 10]{a} shows that $S$ allows a graded version of euclidean division, i.e., for  $f,g\in S^{\hom}$ with $f\neq 0$ there exist unique $u,v\in S^{\hom}$ with $g=uf+v$ and $\deg_{\Z}(v)<\deg_{\Z}(f)$, where $\deg_{\Z}$ denotes the usual $\Z$-degree of polynomials over $R$. Using this we see analogously to \cite[IV.1.7 Proposition 11]{a} that every $G$-graded ideal of $S$ has a homogeneous generating set of cardinality $1$. Next, developing a graded version of the theory of divisibility in entire rings along the line of \cite[VI.1]{a}, it follows analogously to \cite[VI.1.11 Proposition 9 (DIV); VII.1.2 Proposition 1]{a} that for every $x\in Q(S)^{\hom}$ there exist coprime $a,b\in S^{\hom}$ with $x=\frac{a}{b}$. So, the argument in \cite[V.1.3 Proposition 10]{ac} shows that $S$ is integrally closed. As it is noetherian the claim is proven (\ref{a20}).
\end{proof}

\begin{thm}\label{f10}
a) Formation of coarsely graded algebras of torsionfree groups over entire $G$-graded rings commutes with integral closure.

b) If $F$ is a torsionfree group, then an entire $G$-graded ring $R$ is integrally closed if and only if $R[F]_{[G]}$ is so.\footnote{In case $G=0$ the statement that integral closedness of $R$ implies integral closedness of $R[F]_{[G]}$ is \cite[V.1 Exercice 24]{ac}.}
\end{thm}

\begin{proof}
It suffices to prove the first claim. We can without loss of generality suppose that $F$ is of finite type, hence free of finite rank (\ref{p25}, \ref{p40}, \ref{lem3}). By induction on the rank of $F$ we can furthermore suppose $F=\Z$. We have $\inte(R[\Z]_{[G]})\cap Q(R)[\Z]_{[G]}=\inte(R[\Z]_{[G]},Q(R)[\Z]_{[G]})$. Since $Q(R)[\Z]_{[G]}$ is integrally closed (\ref{f20}) we get $$\inte(R[\Z]_{[G]})\subseteq\inte(Q(R)[\Z]_{[G]},Q(R[\Z]_{[G]}))=Q(R)[\Z]_{[G]}.$$ It follows $$\inte(R[\Z]_{[G]})=\inte(R[\Z]_{[G]})\cap Q(R)[\Z]_{[G]}=\inte(R[\Z]_{[G]},Q(R)[\Z]_{[G]})=\inte(R)[\Z]_{[G]}$$ (\ref{a90}) and thus the claim.
\end{proof}

\begin{no}\label{bem}
Let $F$ be a torsionfree group, let $R$ be an entire $G$-graded ring, and suppose that $\cinte(R[\Z]_{[G]})\cap Q(R)[\Z]_{[G]}=\cinte(R[\Z]_{[G]},Q(R)[\Z]_{[G]})$. Then, the same argument as in \ref{f10} (keeping in mind \ref{lem3}) yields $\cinte(R)[F]_{[G]}=\cinte(R[F]_{[G]})$, hence $R$ is completely integrally closed if and only if $R[F]_{[G]}$ is so. However, although $R[\Z]_{[G]}$ is principal by the proof of \ref{f20}, we have seen in \ref{trans10} that it is unclear whether $\cinte(R[\Z]_{[G]})\cap Q(R)[\Z]_{[G]}$ and $\cinte(R[\Z]_{[G]},Q(R)[\Z]_{[G]})$ are equal in general. 
\end{no}

\begin{cor}\label{lem40}
a) Formation of coarsely graded algebras of torsionfree groups over noetherian entire $G$-graded rings commutes with complete integral closure.

b) If $F$ is a torsionfree group, then a noetherian entire $G$-graded ring $R$ is completely integrally closed if and only if $R[F]_{[G]}$ is so.
\end{cor}

\begin{proof}
We can without loss of generality suppose that $F$ is of finite type (\ref{p25}, \ref{p40}, \ref{lem3}). Then, $R[F]_{[G]}$ is noetherian (\ref{p110}), and the claim follows from \ref{f10} and \ref{a20}.
\end{proof}

In the rest of this section we study the behaviour of (complete) integral closures under arbitrary coarsening functors, also using the results from Section 2.

\begin{prop}\label{4.500}
Let $\psi\colon G\twoheadrightarrow H$ be an epimorphism of groups.

a) $\psi$-coarsening commutes with integral closure if and only if a $G$-graded ring $R$ is entire and integrally closed if and only if $R_{[\psi]}$ is so.

b) If $\psi$-coarsening commutes with complete integral closure, then a $G$-graded ring $R$ is entire and completely integrally closed if and only if $R_{[\psi]}$ is so.
\end{prop}

\begin{proof}
If $\psi$-coarsening commutes with (complete) integral closure then it is clear that a $G$-graded ring $R$ is entire and (completely) integrally closed if and only if $R_{[\psi]}$ is so. Conversely, suppose that $\psi$-coarsening preserves the property of being entire and integrally closed. Let $R$ be an entire $G$-graded ring. Since simple $G$-graded rings are entire and integrally closed, $R_{[\psi]}$ is entire (\ref{p90}~b)). As $\inte(R)$ is integrally closed (\ref{a20}) the same is true for $\inte(R)_{[\psi]}$, implying $$\inte(R_{[\psi]})=\inte(R_{[\psi]},Q(R_{[\psi]}))\subseteq\inte(\inte(R)_{[\psi]},Q(R_{[\psi]}))=$$$$\inte(\inte(R)_{[\psi]})=\inte(R)_{[\psi]}\subseteq\inte(R_{[\psi]})$$ (\ref{4.400}) and thus the claim.
\end{proof}

The argument used in a) cannot be used to prove the converse of b), as $\cinte(R)$ is not necessarily completely integrally closed (\ref{4.400}).

\begin{prop}\label{4.600}
Let $\psi\colon G\twoheadrightarrow H$ be an epimorphism of groups. Suppose that $\psi$-coarsening commutes with relative integral closure and maps simple $G$-graded rings to entire and integrally closed $H$-graded rings. Then, $\psi$-coarsening commutes with integral closure.
\end{prop}

\begin{proof}
If $R$ is an entire $G$-graded ring, then $R_{[\psi]}$ is entire (\ref{a120}~b), \ref{p90}~b)) and $Q(R)_{[\psi]}$ is integrally closed, and as $Q(Q(R)_{[\psi]})=Q(R_{[\psi]})$ (\ref{4.400}) it follows $$\inte(R_{[\psi]})=\inte(R_{[\psi]},Q(R_{[\psi]}))\subseteq\inte(Q(R)_{[\psi]},Q(R_{[\psi]}))=\inte(Q(R)_{[\psi]})=Q(R)_{[\psi]},$$ hence $\inte(R_{[\psi]})=\inte(R_{[\psi]},Q(R)_{[\psi]})=\inte(R,Q(R))_{[\psi]}=\inte(R)_{[\psi]}$.
\end{proof}

\begin{no}\label{4.601}
We have seen in \ref{trans10} that it is (in the notations of the proof of \ref{4.600}) not clear that $\cinte(R_{[\psi]})\subseteq Q(R)_{[\psi]}$ implies $\cinte(R_{[\psi]})=\cinte(R_{[\psi]},Q(R)_{[\psi]})$. Therefore, the argument from that proof cannot be used to get an analogous result for complete integral closures.
\end{no}

\begin{lemma}\label{lem50}
Let $F$ be a free direct summand of $G$, let $H$ be a complement of $F$ in $G$, let $\psi\colon G\twoheadrightarrow H$ be the canonical projection, let $R$ be a simple $G$-graded ring, and suppose that $\psi(\degsupp(R))\subseteq\degsupp(R)$. Then, $R_{[\psi]}\cong R_{(H)}[\degsupp(R)\cap F]_{[H]}$ in ${\sf GrAnn}^H$.
\end{lemma}

\begin{proof}
We set $D\dfgl\degsupp(R)$. As $F$ is free the same is true for $D\cap F$. Let $E$ be a basis of $D\cap F$. If $e\in E$ then $R_e\neq 0$, so that we can choose $y_e\in R_e\setminus 0\subseteq R^*$. For $f\in D\cap F$ there exists a unique family $(r_e)_{e\in E}$ of finite support in $\Z$ with $f=\sum_{e\in E}r_ee$, and we set $y_f\dfgl\prod_{e\in E}y_e^{r_e}\in R_f\setminus 0$; in case $f\in E$ we recover the element $y_f$ defined above.

As $(R_{(H)})_{[0]}$ is a subring of $R_{[0]}$ there exists a unique morphism of $(R_{(H)})_{[0]}$-algebras $p\colon R_{(H)}[D\cap F]_{[0]}\rightarrow R_{[0]}$ with $p(e_f)=y_f$ for $f\in D\cap F$. If $h\in H$, then for $f\in D\cap F$ and $x\in R_h$ we have $p(xe_f)=xy_f\in R_{h+f}\subseteq(R_{[\psi]})_h$, so that $p((R_{(H)}[D\cap F]_{[H]})_h)\subseteq(R_{[\psi]})_h$, and therefore we have a morphism $p\colon R_{(H)}[D\cap F]_{[H]}\rightarrow R_{[\psi]}$ in ${\sf GrAnn}^H$.

Let $\chi\colon G\twoheadrightarrow F$ denote the canonical projection. For $g\in G$ with $\chi(g)\in D$ there is a morphism of groups $$\textstyle q_g\colon R_g\rightarrow R_{(H)}[D\cap F],\;x\mapsto\frac{x}{y_{\chi(g)}}e_{\chi(g)},$$ and for $g\in G$ with $\chi(g)\notin D$ we denote by $q_g$ the zero morphism of groups $R_g\rightarrow R_{(H)}[D\cap F]$. For $h\in H$ the morphisms $q_g$ with $g\in\psi^{-1}(h)$ induce a morphism of groups $q_h\colon(R_{[\psi]})_h\rightarrow R_{(H)}[D\cap F]$. So, we get a morphism of groups $$q\dfgl\bigoplus_{h\in H}q_h\colon R_{[\psi]}\rightarrow R_{(H)}[D\cap F].$$

Let $g\in G$ and $x\in R_g$. If $\chi(g)\notin D$ then $g\notin D$, hence $x=0$, and therefore $p(q(x))=x$. Otherwise, $p(q(x))=p(\frac{x}{y_{\chi(g)}}e_{\chi(g)})=\frac{x}{y_{\chi(g)}}p(e_{\chi(g)})=\frac{x}{y_{\chi(g)}}y_{\chi(g)}=x$. This shows that $q$ is a right inverse of $p$. If $x\in R_{(H)}$ then $q(p(x))=x$, and if $f\in D\cap F$ then $q(p(e_f))=q(y_f)=\frac{y_f}{y_f}e_f=e_f$, hence $q$ is a left inverse of $p$. Therefore, $q$ is an inverse of $p$, and thus $p$ is an isomorphism.
\end{proof}

\begin{prop}\label{4.700}
Let $\psi\colon G\twoheadrightarrow H$ be an epimorphism of groups, let $R$ be a simple $G$-graded ring, and suppose that one of the following conditions is fulfilled:
\begin{aufz}
\item[i)] $G$ is torsionfree;
\item[ii)] $\ke(\psi)$ is contained in a torsionfree direct summand of $G$ and $R$ has full support.
\end{aufz}
Then, $R_{[\psi]}$ is entire and completely integrally closed.
\end{prop}

\begin{proof}
First, we note that $R_{[\psi]}$ is entire (\ref{p90}~b)). In case i) it suffices to show that $R_{[0]}$ is integrally closed, so we can replace $H$ with $0$ and hence suppose $\ke(\psi)=G$. In case ii), by the same argument as in the proof of \ref{a120} (and keeping in mind \ref{4.400}) we can suppose without loss of generality that $K\dfgl\ke(\psi)$ itself is a torsionfree direct summand of $G$ and hence consider $H$ as a complement of $K$ in $G$. In both cases, as $K=\varinjlim_{L\in\F_K}L$ (\ref{p25}) we have $G=K\oplus H=\varinjlim_{L\in\F_K}(L\oplus H)$, hence $R=\varinjlim_{L\in\F_K}((R_{(U\oplus H)})^{(G)})$. Setting $\psi_L\dfgl\psi\!\upharpoonright_{L\oplus H}\colon L\oplus H\twoheadrightarrow H$ we get $R_{[\psi]}=\varinjlim_{L\in\F_K}((R_{(L\oplus H)})_{[\psi_L]})$ (\ref{p20}). Hence, if $(R_{(L\oplus H)})_{[\psi]}$ is integrally closed for every $L\in\F_K$ then $R_{[\psi]}$ is integrally closed (\ref{lem3}). Therefore, as $R_{(L\oplus H)}$ is simple for every $L\in\F_K$ (\ref{p50}) we can suppose that $K$ is of finite type, hence free. As $R$ is simple it is clear that $D\dfgl\degsupp(R)\subseteq G$ is a subgroup, hence $D\cap K\subseteq K$ is a subgroup, and thus $D\cap K$ is free. In both cases, our hypotheses ensure $\psi(D)\subseteq D$, so that \ref{lem50} implies $R_{[\psi]}\cong R_{(H)}[D\cap K]_{[H]}$. As $R$ is simple it is completely integrally closed, hence $R_{(H)}$ is completely integrally closed (\ref{4.400}), thus $R_{(H)}[D\cap K]_{[H]}$ is completely integrally closed (\ref{f10}), and so the claim is proven.
\end{proof}

\begin{thm}\label{4.800}
Let $\psi\colon G\twoheadrightarrow H$ be an epimorphism of groups, let $R$ be an entire $G$-graded ring, and suppose that one of the following conditions is fulfilled:
\begin{aufz}
\item[i)] $G$ is torsionfree;
\item[ii)] $\ke(\psi)$ is contained in a torsionfree direct summand of $G$ and $\langle\degsupp(R)\rangle_{\Z}=G$.
\end{aufz}
Then, $\inte(R)_{[\psi]}=\inte(R_{[\psi]})$, and $R$ is integrally closed if and only if $R_{[\psi]}$ is so.\footnote{In case i) and $H=0$ this is \cite[V.1 Exercice 25]{ac}.}
\end{thm}

\begin{proof}
As $\degsupp(Q(R))=\langle\degsupp(R)\rangle_{\Z}$ this follows immediately from \ref{4.600}, \ref{4.700} and \ref{a120}.
\end{proof}

\begin{qu}
Let $R$ be an entire $G$-graded ring. The above, especially \ref{bem} and \ref{4.601}, gives rise to the following questions:
\begin{aufz}
\item[a)] \textit{Let $\psi\colon G\twoheadrightarrow H$ be an epimorphism of groups such that $\ke(\psi)$ is torsionfree. Do we have $\cinte(R_{[\psi]})\cap Q(R)_{[\psi]}=\cinte(R_{[\psi]},Q(R)_{[\psi]})$?}
\item[b)] \textit{Do we have $\cinte(R[\Z]_{[G]})\cap Q(R)[\Z]_{[G]}=\cinte(R[\Z]_{[G]},Q(R)[\Z]_{[G]})$?}
\end{aufz}
If both these questions could be answered positively, then the same arguments as above would yield statements for complete integral closures analogous to \ref{f10}, \ref{4.600}, and \ref{4.800}.
\end{qu}


\textbf{Acknowledgement:} I thank Benjamin Bechtold and the reviewer for their comments and suggestions. The remarks in \ref{a140} were suggested by Micha Kapovich and Will Sawin on {\tt http://mathoverflow.net/questions/108354}. The counterexample in \ref{trans10} is due to an anonymous user on {\tt http://mathoverflow.net/questions/110998}.


\end{document}